\documentclass{elsarticle}
\usepackage{animate}
\usepackage{amsmath}
\usepackage{amssymb}
\usepackage{amsthm}
\usepackage{url}
\usepackage{mathtools,hyperref}
\usepackage{graphics}
\usepackage{subfigure}
\usepackage{graphicx}
\usepackage{textcomp}
\usepackage{mathrsfs}
\usepackage{epstopdf}
\usepackage{algorithmic, algorithm}
\usepackage{braket,amsfonts,amsopn}
\usepackage{color}
\usepackage{accents}
\usepackage{mathscinet}
\usepackage{enumitem}
\usepackage{tikz}
\usepackage{cleveref}
\usepackage{lineno,hyperref}
\modulolinenumbers[5]
\usepackage{multirow}

\usepackage{caption}
\newtheorem{theorem}{Theorem}[section]

\theoremstyle{definition}

\newtheorem{corollary}[theorem]{Corollary}

\theoremstyle{remark}
\newtheorem{remark}[theorem]{Remark}

\numberwithin{equation}{section}

\usepackage{natbib,hyperref}
\begin{document}

\begin{frontmatter}

\title{Gradient recovery for elliptic interface problem: II. immersed finite element methods}

\author[mainaddress]{Hailong Guo}
\ead{hlguo@math.ucsb.edu}

\author[mainaddress]{Xu Yang\corref{mycorrespondingauthor}}
\cortext[mycorrespondingauthor]{Corresponding author}
\ead{xuyang@math.ucsb.edu}

\address[mainaddress]{Department of Mathematics, University of California Santa Barbara, CA, 93106}

\begin{abstract}
This is the second paper on the study of gradient recovery for elliptic interface problem. In our previous work [H. Guo and X. Yang, 2016, arXiv:1607.05898], we developed gradient recovery finite element method based on body-fitted mesh. In this paper,
we propose new gradient recovery methods based on two immersed interface finite element methods: symmetric and consistent immersed finite method [H. Ji, J. Chen and Z. Li, {\it J. Sci. Comput.}, 61 (2014), 533--557] and Petrov-Galerkin immersed finite element method [T.Y. Hou, X.-H. Wu and Y. Zhang, {\it Commun. Math. Sci.}, 2 (2004), 185--205, and S. Hou and X.-D. Liu, {\it J. Comput. Phys.}, 202 (2005), 411--445].  Compared to body-fitted mesh based gradient recover methods, immersed finite element methods provide a uniform way of recovering gradient on regular meshes. Numerical examples are presented to confirm the superconvergence of both gradient recovery methods. Moreover, they provide asymptotically exact {\it a posteriori} error estimators for both immersed finite element methods.
\end{abstract}


\end{frontmatter}


\section{Introduction}

We are interested in developing gradient recovery methods for the following elliptic interface problem
\begin{align}
  -\nabla \cdot (\beta(z) \nabla u(z)) &= f(z),  \quad z \text{ in } \Omega\setminus \Gamma, \label{eq:model}\\
   u & = 0, \quad\quad\,\, z \text{ on } \partial\Omega, \label{eq:bnd}
\end{align}
where $\Omega$ is a bounded polygonal domain with Lipschitz boundary
$\partial \Omega$ in $\mathbb{R}^2$, and $\Gamma$ is the interface which spits $\Omega$ into two disjoint subdomains $\Omega^-$ and $\Omega^+$.
Note that the interface $\Gamma$ can be given by a zero level set of level set function \cite{Osher2003, Sethian1996}.

The interface problem is characterized by the following piecewise smooth diffusion coefficient $\beta(z) \ge \beta_0$,
\begin{equation}\label{eq:test}
\beta(z) =
\left\{
\begin{array}{ccc}
    \beta^-(z) &  \text{if } z\in \Omega^-, \\
   \beta^+(z)  &   \text{if } z\in \Omega^+,
\end{array}
\right.
\end{equation}
which has a finite jump of function value across the interface $\Gamma$. We consider homogeneous jump conditions at the interface $\Gamma$ as below, 
\begin{align}
   [u]_{\Gamma} &= u^+-u^-=0, \label{eq:valuejump}\\
   [\beta \partial_nu]_{\Gamma} &= \beta^+u_n^+ - \beta^-u^-_n = 0, \label{eq:fluxjump}
\end{align}
where $\partial_nu=\nabla u\cdot n$ denotes the normal flux  with $n$ being the unit outer normal vector of the interface $\Gamma$.

Simulation of the interface problem \eqref{eq:model}--\eqref{eq:fluxjump} is an important problem in the fields of fluid dynamics and material science, where background is composed by rather different materials. Discontinuities of coefficients at interface lead to nonsmooth solutions in general, and thus raise a challenge for designing efficient numerical methods for \eqref{eq:model}--\eqref{eq:fluxjump}.

Two mainstreams of existing numerical methods for \eqref{eq:model}--\eqref{eq:fluxjump} are body-fitted mesh-based methods and immersed interface methods. Body-fitted mesh-based methods resolve discontinuities by generating mesh grids to align with interface, and then use standard finite element methods. This type of methods can provide high order accuracy, with nearly optimal error estimates established in, for example, \cite{Babuska1970,BrambleKing1996,ChenZou1998,Xu1982}. Despite its merit of accuracy, a main drawback of such methods
is the requirement of a body-fitted mesh generator, which can be technically involved and time consuming especially when the geometry
of interface becomes complicated. Therefore, it will be more convenient to develop numerical methods based unfitted mesh (e.g. Cartesian mesh).
A rich literature can be found in this direction including immersed boundary method (IBM) by Peskin \cite{Peskin1977, Peskin2002} and immersed interface method (IIM) by Leveque and Li \cite{LevequeLi1994}, just to name a few.

In IBM, Dirac $\delta$-function is used to model discontinuity and discretized to distribute a singular source to nearest grid point. In IIM, a special finite difference scheme is constructed near interface to get an accurate approximation of the solution. Moreover, IIM was also developed in the framework of finite element method \cite{Li1998, LiLinLin2004, LiLinWu2003}. Interested readers are referred to \cite{LiIto2006} for a review of this type of methods. In \cite{LiLinWu2003}, Li, Lin and Wu proposed a nonconforming immersed finite element method (IFEM) by modifying the basis functions on elements
crossing interface. They also established optimal error estimates in $L^2$ and $H^1$ norms in \cite{ChouwKwakWee2010}. However, it only achieved first order (suboptimal) convergence
in $L^{\infty}$ norm due to discontinuities of test functions. To overcome this drawback, Ji, Chen, and Li added a correction term into the bilinear form of the nonconforming IFEM to penalize the discontinuities at interface \cite{JiChenLi2014}, which showed optimal convergence rate in $L^2$ and $H^1$ norms. They also numerically verified that the method achieved second order convergence in $L^{\infty}$ norm.  Another weak form formulation was derived in \cite{HouLiu2005,HouSongWangZhao2013,HouWuZhang2004} based on Petrov-Galerkin method for the discretization of elliptic interface problem, which has been numerically verified to have optimal convergence rate in $L^2$, $H^1$ and
$L^{\infty}$ norms.

Superconvergence analysis of elliptic interface problem has been a challenging problem due to the of lack regularity of solution at interface. Standard gradient recovery methods \cite{ZZ1992a, ZZ1992b, ZhangNaga2005,
NagaZhang2005, AinsworthOden2000, GuoZhang2015} only work well for elliptic problems with smooth coefficient. As far as we know, only limited work has been done in the development of gradient recovery methods for elliptic interface problem. For example, \cite{Chou2012,Chou2015} proposed two special interpolation formula to recover flux for linear and quadratic immersed finite element method in one-dimension.
A more recent work \cite{WeiChenHuangZheng2014} showed a  supercloseness between finite element solution and linear interpolation of the true solution for linear finite element method based on body-fitted mesh. In our previous work \cite{GuoYang2016}, we developed an immerse polynomial preserving recovery (IPPR) method based on body-fitted mesh and proved its superconvergence for both mildly unstructured and adaptive refined meshes.

As a continuous study of \cite{GuoYang2016}, we propose new gradient recovery methods in this paper based on
two immersed finite element methods: symmetric and consistent immersed finite element (SCIFEM)  \cite{JiChenLi2014} and
 Petrov-Galerkin immersed finite element method (PGIFEM) \cite{HouLiu2005,HouSongWangZhao2013,HouWuZhang2004}. The development of the methods is based on the following two observations: firstly, the solution is piecewise smooth on each subdomain despite of its low global regularity;  secondly, finite element solution is discontinuous at interface even though the exact solution is continuous. Accordingly, we design the gradient recovery methods by two steps: enriching and smoothing. We first define an enriching operator to enrich the discontinuous finite element solution into continuous one on a local body-fitted
 mesh obtained by adding extra nodes \cite{LiLinWu2003}. Such type of
 enriching operator has been well studied for nonconforming finite element and plays an important role in \textit{a priori} error estimates \cite{Gudi2010} and convergence analysis of
 multigrid methods \cite{Brenner1996, Brenner1999, Brenner2003}. Then we apply the IPPR gradient recovery operator developed in \cite{GuoYang2016} to the enriched finite element solution. We prove that the proposed gradient recovery operator is a bounded linear operator, and numerically verify that the recovered gradient is $\mathcal{O}(h^{1.5})$ superconvergent to exact gradient. As a byproduct, we observe the $\mathcal{O}(h^{1.5})$ supercloseness
 between finite element solution and linear interpolation of true solution for both  SCIFEM  \cite{JiChenLi2014} and
 PGIFEM \cite{HouLiu2005,HouSongWangZhao2013,HouWuZhang2004}.

The rest of the paper is organized as follows. In Section~2, we briefly review two immersed finite element methods, SCIFEM and PGIFEM, as a preparation for designing gradient recovery methods. In Section~3, we first define
an enriching operator and prove several properties of the operator. Then, we propose the gradient recovery methods for SCIFEM and PGIFEM
and prove that the gradient recovery operator is a linear, bounded and consistent operator.  In Section~4, serval numerical examples are presented to confirm the superconvergence of the gradient recovery method. We make conclusive remarks in Section~5.

\section{Review on immersed finite element methods}
In this section, we briefly review two immersed finite element methods, symmetric and consistent immersed finite element method \cite{JiChenLi2014} and Petrov-Galerkin immersed finite element method \cite{HouLiu2005,HouSongWangZhao2013,HouWuZhang2004}, based on which we shall develop superconvergent gradient recovery methods for elliptic interface problem \eqref{eq:model}--\eqref{eq:fluxjump} in Section~3.

\subsection{Notations}
We first summarize the notations that will be used in this paper.
We will use standard
notations for Sobolev spaces and their associate norms given in \cite{BrennerScott2008, Ciarlet2002, Evans2008} .
 For a subdomain $A$
of $\Omega$, let $\mathbb{P}_m(A)$ be the space of polynomials of
degree less than or equal to $m$ in $A$ and $n_m$ be the
dimension of $\mathbb{P}_m(A)$ which equals to $\frac{1}{2}(m+1)(m+2)$.
$W^{k,p}(A)$ denotes the Sobolev space with norm
$\|\cdot\|_{k, p, A} $ and seminorm $|\cdot|_{k, p,A}$.
 When $p = 2$, $W^{k,2}(A)$ is simply  denoted by $H^{k}(A)$
 and the subscript $p$ is omitted in its associate norm and seminorm.
  As in \cite{WeiChenHuangZheng2014}, denote
 $W^{k,p}(\Omega^-\cup\Omega^+)$  as the function space consisting of piecewise Sobolev  function $w$  such
 that $w|_{\Omega^-}\in W^{k,p}(\Omega^-)$ and $w|_{\Omega^+}\in W^{k,p}(\Omega^+)$.  For the function
 space $W^{k,p}(\Omega^-\cup\Omega^+)$, define its associated norm  as
 \begin{equation*}
\|w\|_{k,p, \Omega^-\cup\Omega^+} = \left( \|w\|_{k,p, \Omega^-}^p + \|w\|_{k,p, \Omega^+}\right)^{1/p},
\end{equation*}
and associated seminorm as
 \begin{equation*}
|w|_{k,p, \Omega^-\cup\Omega^+} = \left( |w|_{k,p, \Omega^-}^p + |w|_{k,p, \Omega^+}\right)^{1/p}.
\end{equation*}

Let $C$ denote a generic positive constant which may be different at different occurrences.
For the sake of simplicity, we use $x\lesssim y$ to mean that $x\leq Cy$ for some constant $C$
independent of mesh size.

Without loss of generality,  we simply suppose   $\mathcal{T}_h$  is a uniform triangulation of $\Omega$
with $h = \mbox{diam}(T)$.   Assume $h$ is small enough so that the interface $\Gamma$ never crosses any edge
of $\mathcal{T}_h$ more than two times.
The elements of $\mathcal{T}_h$ can be divided into categories : regular element and interface element.
We call an element $T$ interface element if  the interface $\Gamma$ passes the interior of $T$; otherwise
we call it regular element.  Remark that if $\Gamma$ only passes two vertices of an element $T$, we treat
the element $T$ as a regular element. Let $\mathcal{T}^{i}_h$ and $\mathcal{T}^{r}_h$ denote
the set of all interface elements and regular elements respectively.  The set of all vertices of
$\mathcal{T}_h$ is denoted by $\mathcal{N}_h$.

\subsection{Variational formula}
The variational formulation to elliptic interface problem \eqref{eq:model}--\eqref{eq:fluxjump} is given by finding $u \in H^1_0(\Omega)$ such that
 \begin{equation}
(\beta\nabla u, \nabla v)  = (f, v) , \quad \forall v \in H^1_0(\Omega),
\label{eq:var}
\end{equation}
where $(\cdot, \cdot)$  is  standard $L_2$-inner product in the spaces $L^2(\Omega)$.
By the positiveness of $\beta$, Lax-Milgram  Theorem implies \eqref{eq:var} has  a unique solution.
\cite{ChenZou1998, RS1969} proved that $u\in H^{r}(\Omega^-\cup\Omega^+)$ for $0 \le r\le 2$ and
\begin{equation}
\|u\|_{r, \Omega^-\cup\Omega^+} \lesssim \|f\|_{0, \Omega} + \|g\|_{r-3/2, \Gamma},
\end{equation}
if $f\in L^2(\Omega)$ and $g \in H^{r- 3/2}(\Gamma)$.

%
%
%
%
\subsection{Immersed finite element methods}


The key idea of immersed interface methods is to
construct special basis functions in interface elements to incorporate jump conditions \eqref{eq:valuejump} and \eqref{eq:fluxjump}.
As an illustration, we consider a typical interface element $T$ as in Figure \ref{fig:intelem}.
Let $z_4$ and $z_5$ be the intersection points between the interface $\Gamma$ and edges of the element.  Connect the line segment
$\overline{z_4z_5}$ and it forms an approximation of interface $\Gamma$ in the element $T$, denoted by $\Gamma_h|_{T}$.
Then the element $T$ is spitted  into two parts:  $T^-$ and $T^+$.   The special basis $\phi_i$ on the interface element $T$ is constructed
as the following  piecewise linear function
\begin{equation}\label{eq:ifembasis}
\phi_i(z)=
\left\{
\begin{array}{ccc}
\phi^+_i  = a^++b^+x+c^+y, & z = (x, y)\in T^+,    \\
 \phi^-_i= a^-+b^-x+c^-y, & z = (x, y)\in T^-,      \\
\end{array}
\right.
\end{equation}
where the coefficients are determined by the following linear system
\begin{align}
 &\phi_i(z_1) = \delta_{i1}, \, \phi_i(z_2) = \delta_{i2},\, \phi_i(z_3) = \delta_{i3},\label{eq:valueeq}\\
& \phi^+_i(z_4) = \phi^-_i(z_4), \, \phi^+_i(z_5) = \phi^-_i(z_5), \,
\beta^+\partial_n\phi^+_i = \beta^-\partial_n\phi^-_i, \label{eq:fluxeq}
\end{align}
for $i = 1, 2, 3$.  The immersed finite element space $V_h$  \cite{LiLinWu2003} is defined as
\begin{align}
&V_h := \left\{v\in V_h:  v|_{T} \in V_h(T) \text{ and }  v \text{ is continuous on } \mathcal{N}_h, \right\},\\
&V_{h,0} = \left\{v\in V_h:   v(z) = 0  \text{ for all } z \in \mathcal{N}_h\cap \partial\Omega \right\},
\end{align}
where
\begin{equation}
V_h(T) :=
\left\{
\begin{array}{ll}
 \left\{v| v \in \mathbb{P}_1(T) \right\},&   \text{if } T \in \mathcal{T}^{r}_h;    \\
  \left\{v| v \text{ is defined by } \eqref{eq:ifembasis}-\eqref{eq:fluxeq}\right\},&   \text{if } T \in \mathcal{T}^{i}_h .   \\
\end{array}
\right.
\end{equation}
Note that in general $V_h$ is a nonconforming finite element space and
\cite{LiLinLin2004} shows it has optimal approximation capability.

\begin{figure}[ht]
    \centering
    \includegraphics[width=0.5\textwidth]{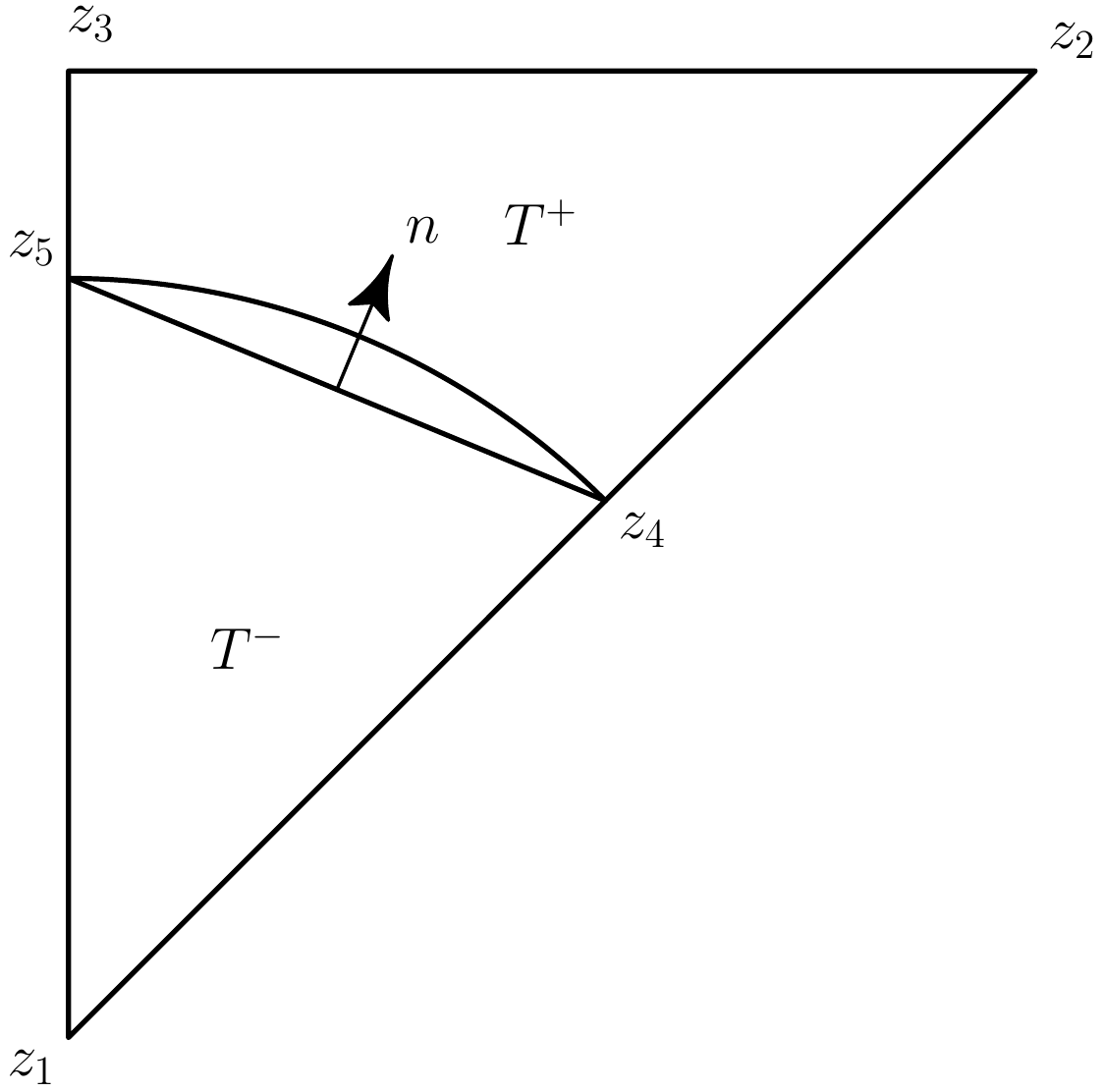}
    \caption{ Typical example of interface element.}
\label{fig:intelem}
\end{figure}

\subsubsection{Symmetric and consistent immersed finite element method}
Let $\mathcal{E}_h$ denote the set of all edges in $\mathcal{T}_h$, and then $\mathcal{E}_h$
consists of interface edge $\mathcal{E}_h^{i}$ and regular edge
$\mathcal{E}_h^{r}$, defined by
\begin{equation}
\mathcal{E}_h^{i} = \{e \in \mathcal{E}_h:  \mathring{e}\cap \Gamma \neq \emptyset\},
\,, \mathcal{E}_h^{r} = \mathcal{E}_h \setminus \mathcal{E}_h^{i}.
\end{equation}
For any interior edge $e$,  there exist two triangles $T_1$ and $T_2$ such that $T_1\cap T_2 = e$.
Denote $n_e$ as the unit normal of $e$ pointing from $T_1$ to $T_2$, and define
\begin{align}
 &\left\{\nabla u\right\} = \frac{1}{2}\left( \nabla u|_{T_1} + \nabla u|_{T_2}\right),\label{eq:edgeaverage}\\
 &[u] = u|_{T_1} - u|_{T_2}. \label{eq:edgejump}
\end{align}
The symmetric and consistent immersed finite element method (SCIFEM) \cite{JiChenLi2014} seeks
$u_h^{sc} \in V_{h,0}$ such that
\begin{equation}\label{eq:scifem}
a_h^{sc}(u_h^{sc}, v_h) = (f, v_h) , \quad \forall v_h \in V_{h,0},
\end{equation}
where
\begin{equation}\label{eq:scbilinear}
a_h^{sc}(u, v) = \sum_{T\in\mathcal{T}_h} \int_T \beta \nabla u \cdot \nabla vdx
+\sum_{e\in \mathcal{E}_h^{i} }\int_e (\{\beta  \nabla u\} [u] + \{\beta \nabla v\}[u])\cdot n_eds
\end{equation}

In  \cite{JiChenLi2014}, Ji, Chen, and Li  showed the bilinear form \eqref{eq:scbilinear} was consist and numerically 
verified its   coercivity.
 Moreover, \cite{JiChenLi2014} proved the following convergence results:
\begin{theorem}
 Let $u$ be the solution of \eqref{eq:model}--\eqref{eq:fluxjump} and $u_h$ be the
 solution of \eqref{eq:scifem}. Then the following error estimates  hold:
\begin{align}
 &\left( \sum_{T\in \mathcal{T}_h} |u-u_h^{sc}|_{H^1(T)}^2\right)^{1/2} \lesssim h \|u\|_{2, \Omega^+\cup\Omega^-},\\
  &\|u-u_h^{sc}\|_{0,\Omega} \lesssim h \|u\|_{2, \Omega^+\cup\Omega^-}.
\end{align}
\end{theorem}

\begin{remark}
 The main difference between SCIFEM and classical immersed finite element method \cite{LiLinWu2003} is that the bilinear form of SCIFEM \eqref{eq:scbilinear} contains
one more term to penalize the discontinuous of basis function at the intersecting points of interface and edge.
Numerical results in \cite{JiChenLi2014} show that SCIFEM has $O(h^2)$ convergence in $L^{\infty}$-norm.
\end{remark}

\subsubsection{Petrov-Galerkin immersed finite element method}
Denote the standard $C^0$ linear finite element space on $\mathcal{T}_h$
by $S_h$ and $S_{h,0} = S_h \cap H^1_0(\Omega)$.
Then the Petrov-Galerkin immersed finite element method (PGIFEM)  \cite{HouWuZhang2004,HouLiu2005,HouSongWangZhao2013} is to find
$u_h^{pg} \in V_{h,0}$  such that
\begin{equation}\label{eq:pvifem}
a_h(u_h^{pg}, v_h) = (f, v_h) , \quad \forall v_h \in S_{h,0},
\end{equation}
where
\begin{equation}
a_h(u, v) = \sum_{T\in\mathcal{T}_h} \int_T \beta \nabla u \cdot \nabla vdx.
\end{equation}

\begin{remark}
 To our best knowledge, there has been no analytical results on estimating PGIFEM, however,
 plenty of numerical simulations indicate that it can achieve optimal convergence rate in both $L_2$, $H_1$ and $L_{\infty}$ norms \cite{HouWuZhang2004,HouLiu2005,HouSongWangZhao2013}.
\end{remark}

\section{Gradient recovery for immersed finite element methods}
In the section, we systematically introduce gradient recovery methods for SCIFEM and PGIFEM reviewed in last section. We first define an enriching operator, and then apply the immersed polynomial preserving recovery operator \cite{GuoYang2016}
to the enriched finite element solution.
\subsection{Enriching operator}
To define the enriching operator, one needs to generate a local body-fitted  mesh $\widehat{\mathcal{T}}_h$ based on
$\mathcal{T}_h$ by adding new vertices into $\mathcal{N}_h$ which divides interface element into three subtriangles.
   Then the new triangulation is constructed as below \cite{LiLinWu2003}:
\begin{enumerate}
\item Keep all regular elements unchanged.
\item For each interface element $T$,  split it into a small  triangle and a quadrilateral
by connecting two intersection points, and then divide the quadrilateral into two subtriangles  by an
auxiliary line connecting a vertex and an intersection point. The choice of auxiliary line is made
so that there at least exists one angle between $\frac{\pi}{4}$ and $\frac{3\pi}{4}$ in the two new
subtriangles.
\end{enumerate}

\begin{remark}
 Note that the new triangulation can contain narrow triangles, and thus standard linear finite
 element method deteriorates on $\widehat{\mathcal{T}}_h$. However, the propose of introducing
 the body-fitted mesh $\widehat{\mathcal{T}}_h$ is just for enriching existing immersed finite element
 solution instead of solving interface problem directly on it.
 \end{remark}

Let $X_h$ be the $C^0$ linear  finite element  space defined on $\widehat{\mathcal{T}}_h$.  We construct
an enriching operator $E_h: V_h \rightarrow X_h$ by averaging the discontinuous values at intersection
points.   Let $\widehat{\mathcal{N}}_h$ denote all vertices in $\widehat{\mathcal{T}}_h$, and one has
$\mathcal{N}_h \subset \widehat{\mathcal{N}}_h$.  For any $z \in \widehat{\mathcal{N}}_h$,
let $\widehat{\mathcal{T}}_z$ denote the set of all triangles in $\widehat{\mathcal{T}}_h$ having $z$
as their vertex and  define
\begin{equation}\label{eq:enrich}
(E_h v)(z) = \frac{1}{|\widehat{\mathcal{T}}_z|} \sum_{\widehat{T} \in \widehat{\mathcal{T}}_z}
v_{\widehat{T}}(z),
\end{equation}
with $|\widehat{\mathcal{T}}_z|$ being the cardinality of $\widehat{\mathcal{T}}_z$ and $v_{\widehat{T}} = v|_{\widehat{T}}$.
We can define $E_hv$ on $\Omega$ by standard linear finite element
interpolation in $X_h$ after obtaining the values $(E_hv)(z)$ at all vertices.
It is easy to see that $(E_hv)(z) = v(z)$ for all $z \in  \widehat{\mathcal{N}}_h\cap
\mathcal{N}_h$, which means  $(E_hv)(z) = v(z)$ for all $z \in  \widehat{\mathcal{N}}_h$
provided that $v$ is continuous.

\begin{remark}
The purpose of the enriching operator is to make  the discontinuous immersed finite element solution become continuous as the true solution.
 \end{remark}

For the enriching operator $E_h$, we can prove the following error estimate.
\begin{theorem}
 For any $v\in V_h$, one has
 \begin{equation}\label{eq:estimate}
\sum_{T\in \widehat{\mathcal{T}}_h}
\|E_hv - v\|^2_{0, T}  \lesssim
h^2\sum_{T\in \mathcal{T}_h}|v|_{1, T}^2.
\end{equation}
\end{theorem}
\begin{proof}
 For any $z \in  \widehat{\mathcal{N}}_h\setminus\mathcal{N}_h$,   there exists an $e\in \mathcal{E}_h^{i}$  so that $z \in e$.
 Let $T_1 $ and $T_2$ be the two triangles  in $\mathcal{T}_h$
 so that $T_1\cap T_2 = e$. Then $\widehat{T}\subset T_1$ or $\widehat{T}\subset T_2$  for any $\widehat{T} \in \widehat{\mathcal{T}}_z$.
 Hence $ v_{\widehat{T}}(z) = v_{T_1}(z)$ or  $ v_{\widehat{T}}(z) = v_{T_2}(z)$.
 Then for any $\widehat{T}_a, \widehat{T}_b \in \widehat{\mathcal{T}}_z$,  one has
 \begin{equation}\label{eq:scale}
\begin{split}
 &[v_{\widehat{T}_a}(z) - v_{\widehat{T}_b}(z) ]^2 \\
 \leq &[v_{T_1}(z) - v_{T_2}(z) ]^2\\
 \leq &[v_{T_1}(z) - v_{T_1}(p) ]^2 + [v_{T_1}(p) - v_{T_2}(p) ]^2 + [v_{T_2}(p) - v_{T_2}(z) ]^2\\
 = & [v_{T_1}(z) - v_{T_1}(p) ]^2 +  [v_{T_2}(p) - v_{T_2}(z) ]^2\\
 \lesssim & |v|^2_{1, T_1\cup T_2},
 \end{split}
\end{equation}
where $p\in \mathcal{N}_h$ and  we have used the mean value theorem \cite{Brenner2003} in the last inequality.

Combining  \eqref{eq:enrich} and \eqref{eq:scale} gives, for any $\widehat{T} \in \widehat{\mathcal{T}}_z$,
\begin{equation}
[(E_hv-v_{\widehat{T}})(z)]^2 \lesssim  |v|^2_{1, T_1\cup T_2},  \forall v \in V_h,
\end{equation}
which implies that
\begin{equation}
\begin{split}
 \|E_hv-v\|^2_{0, \widehat{T}} & \leq |\widehat{T}| \sum_{z \in \mathcal{N}(\widehat{T})}[(E_hv-v_{\widehat{T}})(z)]^2\\
 &\lesssim h^2  \sum_{T \in \mathcal{T}(\widehat{T})}|v|^2_{1, T},
\end{split}
\end{equation}
where $\mathcal{T}(\widehat{T}) = \{ T \in \mathcal{T}_h:  T \cap \widehat{T} \neq \emptyset\}$.
Taking summation over all $\widehat{\mathcal{T}}_h$ produces the inequality  \eqref{eq:estimate}.
  \end{proof}

\begin{corollary}\label{cor:bound}
 For any $v \in V_h$,  we have
\begin{align}
 &\|E_hv\|_{0, \Omega}\lesssim \|v\|_{0, \Omega}\label{eq:l2bound},\\
 &| E_hv|_{1, \Omega} \lesssim | v|_{1, \Omega}\label{eq:h1bound}.
\end{align}
\end{corollary}

\begin{proof}
 We first prove the inequality \eqref{eq:l2bound}.  Notice that
 \begin{equation}
\begin{split}
 \|E_hv\|_{0, \Omega}&\lesssim \|E_hv-v\|_{0, \Omega} + \|v\|_{0, \Omega}\\
 &\lesssim  h \|\nabla v\|_{0, \Omega} + \|v\|_{0, \Omega}\\
 &\lesssim  \|v\|_{0, \Omega},
 \end{split}
\end{equation}
where we have used the standard inverse estimate \cite{Ciarlet2002, BrennerScott2008} in the last inequality.
Using \eqref{eq:estimate} and standard  inverse estimate yields
\begin{equation}
\begin{split}
 |E_hv|_{1, \Omega} &\le |E_hv-v|_{1, \Omega} + |v|_{1, \Omega} \\
 &\lesssim h^{-1}\|E_hv-v\|_{0, \Omega} + |v|_{1, \Omega}\\
 &\lesssim |v|_{1, \Omega},
\end{split}
\end{equation}
which completes our proof.
\end{proof}
\subsection{Gradient Recovery Operator}
The edges of $\widehat{\mathcal{T}}_h$ with both ending points lying on $\Gamma$ form an approximation of the interface $\Gamma$, denoted by $\Gamma_h$, then the triangulation $\widehat{\mathcal{T}}_h$ is divided into the following two disjoint sets by $\Gamma_h$:
\begin{align}
 &\widehat{\mathcal{T}}^-_h:=\left\{  T\in \mathcal{T}_h| \text{ all three vertices of  } T \text{ are in  }\overline{\Omega^-}  \right\},\label{equ:mmesh}\\
 & \widehat{\mathcal{T}}^+_h:=\left\{  T\in \mathcal{T}_h| \text{ all three vertices of  } T \text{ are in  }\overline{\Omega^+}  \right\} . \label{equ:pmesh}
\end{align}
Suppose $X_h^-$ and $X_h^+$ are the continuous linear finite element spaces defined on $ \widehat{\mathcal{T}}^-_h$ and $ \widehat{\mathcal{T}}^+_h$ respectively.

Let $G_h^I: X_h \rightarrow (X_h^-\cup X_h^+) \times (X_h^-\cup X_h^+)$ be the immersed polynomial preserving recovery (IPPR) operator
introduced in \cite{GuoYang2016}.  Let $u_h$ be the solution of  either  symmetric and consistent immersed finite element method or  Petrov-Galerkin immersed finite element method.
The recovered gradient of $u_h$ is defined as
\begin{equation}\label{eq:gr}
R_h u_h = G_h^I(E_hu_h).
\end{equation}

\begin{remark}
 The proposed gradient recovery method consists of two steps:  firstly, we enrich the immersed finite element solution by the enriching operator; then we recover the gradient of the enriched solution.
\end{remark}
It is easy to see that $R_h$ is a linear operator from $V_h$ to $(X_h^-\cup X_h^+) \times (X_h^-\cup X_h^+)$, and one can prove the following boundedness results.
\begin{theorem}\label{thm:grbd}
Denote $R_h$ to be the recovered operator defined in \eqref{eq:gr}, and then
 \begin{equation}
\|R_hu_h\|_{0, \Omega^-\cup \Omega^+} \lesssim |u_h|_{1, h}.
\end{equation}
\end{theorem}
\begin{proof}
 By the definition of IPPR recovery operator in \cite{GuoYang2016}, we have
 \begin{equation*}
\|R_hu_h\|_{0, \Omega^-} = \|G_h^IE_hu_h\|_{0, \Omega^-}  \lesssim  |E_hu_h|_{1, \Omega^-},
\end{equation*}
and
 \begin{equation*}
\|R_hu_h\|_{0, \Omega^+} = \|G_h^IE_hu_h\|_{0, \Omega^+}  \lesssim  |E_hu_h|_{1, \Omega^+}.
\end{equation*}
Then the estimate follows by that
\begin{equation}
\begin{split}
 \|R_hu_h\|_{0, \Omega^-\cup \Omega^+} \le & \|R_hu_h\|_{0, \Omega^-} + \|R_hu_h\|_{0, \Omega^-}\\
\lesssim  &|E_hu_h|_{1, \Omega^-} +  |E_hu_h|_{1, \Omega^+}\\
\lesssim  &|E_hu_h|_{1, \Omega}\\
\lesssim  &|u_h|_{1, \Omega},
\end{split}
\end{equation}
where we have used Corollary \ref{cor:bound}.
\end{proof}

Theorem \ref{thm:grbd} implies $R_h$ is a linear bounded operator. Moreover, we have the following consistency result:
\begin{theorem}\label{thm:grcst}
 Let $R_h: V_h \rightarrow (X_h^-\cup X_h^+)\times (X_h^-\cup X_h^+)$ be the gradient recovery operator defined in \eqref{eq:gr}. Given
 $u\in H^{3}(\Omega^-\cup\Omega^+) \cap C^0(\Omega)$, one has
 \begin{equation}
\|R_hu_I - \nabla u \|_{0, \Omega} \lesssim h^{2}\|u\|_{3, \Omega^-\cup\Omega^+},
\label{equ:grbd}
\end{equation}
where $u_I$ is interpolation of $u$ into linear finite element space $X_h$.
\end{theorem}
\begin{proof}
 Since $u \in C^0(\Omega)$, one has that $u_I \in C^0(\Omega)$ and then $E_hu_I = u_I$.
 Therefore, we have $R_hu_I= G_h^IE_hu_I = G_hu_I$. Theorem  3.6 in \cite{GuoYang2016} implies that
  \begin{equation*}
\|R_hu_I - \nabla u \|_{0, \Omega}  = \|G_h^Iu_I - \nabla u \|_{0, \Omega}
\le h^{2}\|u\|_{3, \Omega^-\cup\Omega^+},
\end{equation*}
which completes our proof.
\end{proof}

\begin{remark} Theorem \ref{thm:grcst} implies $R_h$ is consistent. In addition, it is a local gradient recovery operator. Therefore, $R_h$
satisfies the three conditions of a good gradient recovery operator described in \cite{AinsworthOden2000}, and should serve as an
ideal candidate of gradient recovery operator for both SCIFEM and PGIFEM.
\end{remark}

\begin{remark}  One of the most practical applications of gradient recovery techniques is to construct asymptotically exact {\it a posteriori} error estimators \cite{AinsworthOden2000,Babuska2001,Guo2016,NagaZhang2004, ZZ1992a,ZZ1992b}
for adaptive computational methods.   Based on the recovery operator $R_h$, one can define a local {\it a posteriori} error
estimator on element $T\in \mathcal{T}_h$ as
\begin{equation*}\label{equ:localind}
\eta_T =
\left\{
\begin{array}{lcc}
    \|\beta^{1/2}(R_hu_h - \nabla u_h)\|_{0, T}, &  \text{if } T \in \mathcal{T}_h^r, \\
   \left(\sum\limits_{\widehat{T}\subset T, \widehat{T}\in \widehat{T}_h} \|\beta^{1/2}(R_hu_h - \nabla u_h)\|_{0, \widehat{T}}^2\right)^{\frac{1}{2}}, &   \text{if } T \in \mathcal{T}_h^i,
\end{array}
\right.
\end{equation*}
and the corresponding global error estimator  as
\begin{equation*}\label{equ:globalind}
\eta_h = \left( \sum_{T\in \mathcal{T}_h}\eta_T^2\right)^{1/2},
\end{equation*}
which provides an asymptotically exact {\it a posteriori} error estimator for SCIFEM and PGIFEM.  The readers are referred to \cite{ChenXiaoZhang2009,WuLiLai2011} for residual-type {\it a posteriori} error estimator for immersed
finite element methods.
\end{remark}

\section{Numerical Results}  In the section, we give serval numerical examples to verify the superconvergence
of gradient recovery methods for both SCIFEM and PGIFEM.  The computational domain of all the examples
are chosen as $\Omega = [-1, 1]\times [-1, 1]$.  The uniform triangulation of $\Omega$ is obtained by dividing $\Omega$ into $N^2$
subsquares and then dividing each subsquare into two right triangles.  In all the following tests, we take $N=2^k$ with $k = 5, 6, 7, 8, 9, 10, 11$.
For convenience, we shall use the following error norms in all examples:
\begin{equation}
De:=\|u-u_h\|_{1,\Omega},\quad D^ie:=\|\nabla u_I- \nabla u_h\|_{0, \Omega},\quad
D^re:=\|\nabla u-R_hu_h\|_{0, \Omega}.\\
 \end{equation}

{\bf Example 4.1.} In this example, we  consider  the elliptic interface problem  \eqref{eq:model} with a circular interface of radius $r_0 = 0.6$ as studied in \cite{LiLinWu2003}.
 The exact solution is
\begin{equation*}
u(z) =
\left\{
\begin{array}{ll}
    \frac{r^3}{\beta^-}   &  \text{if }   z\in \Omega_-, \\
      \frac{r^3}{\beta^+} + \left( \frac{1}{\beta^-}-\frac{1}{\beta^+} \right)r_0^3&  \text{if } z \in \Omega^+,\\
   \end{array}
\right.
\end{equation*}
where $r = \sqrt{x^2+y^2}$.

Tables \ref{tab:ex1asc}--\ref{tab:ex1cpg} show the numerical results  of both SCIFEM and PGIFEM with three typical different
jump ratios: $\beta^-/\beta^+ = 1/10$ (moderate jump), $\beta^-/\beta^+ = 1/1000$ (large jump),
and $\beta^-/\beta^+ = 1000$ (large jump).   In all different cases,  optimal $\mathcal{O}(h)$ convergence can be observed for $H^1$-semi error
of finite element solution, which consists with the numerical results in  \cite{JiChenLi2014, HouLiu2005, HouWuZhang2004}.  The recovered gradient superconverges to the exact gradient at a rate of $\mathcal{O}(h^{1.5})$.  Moreover, we numerically
observe the supercloseness between gradient of the finite element solution and its finite element interpolation for both SCIFEM and PGIFEM; see column~5 of Tables \ref{tab:ex1asc}--\ref{tab:ex1cpg}.


\begin{table}[htb!]
\centering
\caption{Numerical results of SCIFEM for Example 4.1 with $\beta^+=10, \beta^-=1$. }\label{tab:ex1asc}
\begin{tabular}{|c|c|c|c|c|c|c|c|}
\hline
 $N$ & $De$ & order& $D^{i}e$ & order& $D^{r}_re$ & order\\ \hline\hline
 32 &5.71e-02&--&1.47e-02&--&2.19e-02&--\\ \hline
 64 &2.94e-02&0.96&4.48e-03&1.72&7.48e-03&1.55\\ \hline
 128 &1.47e-02&1.00&1.84e-03&1.28&2.31e-03&1.69\\ \hline
 256 &7.38e-03&0.99&6.46e-04&1.51&7.40e-04&1.64\\ \hline
 512 &3.70e-03&1.00&2.36e-04&1.45&2.93e-04&1.34\\ \hline
 1024 &1.85e-03&1.00&8.11e-05&1.54&1.02e-04&1.52\\ \hline
 2048 &9.26e-04&1.00&2.83e-05&1.52&3.44e-05&1.57\\ \hline
\end{tabular}
\end{table}

\begin{table}[htb!]\label{tab:ex1apg}
\centering
\caption{Numerical results of PGIFEM for Example 4.1 with $\beta^+=10, \beta^-=1$. }
\begin{tabular}{|c|c|c|c|c|c|c|c|}
\hline
 $N$ & $De$ & order& $D^{i}e$ & order& $D^{r}_re$ & order\\ \hline\hline
 32 &5.92e-02&--&2.15e-02&--&3.09e-02&--\\ \hline
 64 &2.98e-02&0.99&6.61e-03&1.71&1.01e-02&1.61\\ \hline
 128 &1.48e-02&1.01&2.59e-03&1.35&3.33e-03&1.61\\ \hline
 256 &7.41e-03&1.00&9.00e-04&1.53&1.08e-03&1.63\\ \hline
 512 &3.71e-03&1.00&3.27e-04&1.46&4.11e-04&1.39\\ \hline
 1024 &1.85e-03&1.00&1.13e-04&1.54&1.43e-04&1.52\\ \hline
 2048 &9.26e-04&1.00&3.99e-05&1.50&4.92e-05&1.54\\ \hline
\end{tabular}
\end{table}

\begin{table}[htb!]\label{tab:ex1bsc}
\centering
\caption{Numerical results of SCIFEM for Example 4.1 with $\beta^+=1000, \beta^-=1$. }
\begin{tabular}{|c|c|c|c|c|c|c|c|}
\hline
 $N$ & $De$ & order& $D^{i}e$ & order& $D^{r}_re$ & order\\ \hline\hline
 32 &5.69e-02&--&2.43e-02&--&2.46e-02&--\\ \hline
 64 &2.77e-02&1.04&3.50e-03&2.79&6.44e-03&1.93\\ \hline
 128 &1.38e-02&1.00&1.61e-03&1.12&1.95e-03&1.72\\ \hline
 256 &6.95e-03&0.99&5.36e-04&1.58&6.34e-04&1.62\\ \hline
 512 &3.49e-03&1.00&1.95e-04&1.46&2.54e-04&1.32\\ \hline
 1024 &1.75e-03&1.00&6.61e-05&1.56&8.76e-05&1.53\\ \hline
 2048 &8.74e-04&1.00&2.29e-05&1.53&2.98e-05&1.55\\ \hline
\end{tabular}
\end{table}

\begin{table}[htb!]\label{tab:ex1cpg}
\centering
\caption{Numerical results of PGIFEM for Example 4.1 with $\beta^+=1000, \beta^-=1$. }
\begin{tabular}{|c|c|c|c|c|c|c|c|}
\hline
 $N$ & $De$ & order& $D^{i}e$ & order& $D^{r}_re$ & order\\ \hline\hline
 32 &5.95e-02&--&3.27e-02&--&4.55e-02&--\\ \hline
 64 &2.91e-02&1.03&9.35e-03&1.81&1.21e-02&1.91\\ \hline
 128 &1.44e-02&1.02&4.03e-03&1.21&4.54e-03&1.41\\ \hline
 256 &7.10e-03&1.02&1.45e-03&1.48&1.48e-03&1.62\\ \hline
 512 &3.53e-03&1.01&5.68e-04&1.35&5.94e-04&1.31\\ \hline
 1024 &1.76e-03&1.01&1.90e-04&1.58&1.95e-04&1.60\\ \hline
 2048 &8.76e-04&1.00&6.80e-05&1.48&6.96e-05&1.49\\ \hline
\end{tabular}
\end{table}

\begin{table}[htb!]\label{tab:ex1csc}
\centering
\caption{Numerical results of SCIFEM for Example 4.1 with $\beta^+=1, \beta^-=1000$. }
\begin{tabular}{|c|c|c|c|c|c|c|c|}
\hline
 $N$ & $De$ & order& $D^{i}e$ & order& $D^{r}_re$ & order\\ \hline\hline
 32 &1.95e-01&--&1.35e-02&--&1.92e-02&--\\ \hline
 64 &9.79e-02&1.00&3.60e-03&1.91&8.14e-03&1.24\\ \hline
 128 &4.90e-02&1.00&1.48e-03&1.28&2.21e-03&1.88\\ \hline
 256 &2.45e-02&1.00&5.56e-04&1.42&7.46e-04&1.57\\ \hline
 512 &1.23e-02&1.00&1.81e-04&1.61&2.38e-04&1.65\\ \hline
 1024 &6.13e-03&1.00&6.44e-05&1.49&8.57e-05&1.48\\ \hline
 2048 &3.06e-03&1.00&2.33e-05&1.47&2.99e-05&1.52\\ \hline
\end{tabular}
\end{table}

\begin{table}[htb!]
\centering
\caption{Numerical results of PGIFEM for Example 4.1 with $\beta^+=1, \beta^-=1000$. }\label{tab:ex1cpg}
\begin{tabular}{|c|c|c|c|c|c|c|c|}
\hline
 $N$ & $De$ & order& $D^{i}e$ & order& $D^{r}_re$ & order\\ \hline\hline
 32 &5.95e-02&--&3.27e-02&--&4.55e-02&--\\ \hline
 64 &2.91e-02&1.03&9.35e-03&1.81&1.21e-02&1.91\\ \hline
 128 &1.44e-02&1.02&4.03e-03&1.21&4.54e-03&1.41\\ \hline
 256 &7.10e-03&1.02&1.45e-03&1.48&1.48e-03&1.62\\ \hline
 512 &3.53e-03&1.01&5.68e-04&1.35&5.94e-04&1.31\\ \hline
 1024 &1.76e-03&1.01&1.90e-04&1.58&1.95e-04&1.60\\ \hline
 2048 &8.76e-04&1.00&6.80e-05&1.48&6.96e-05&1.49\\ \hline
\end{tabular}
\end{table}

\begin{figure}[ht]
\centering
\subfigure[]{%
     \includegraphics[width=0.43\textwidth]{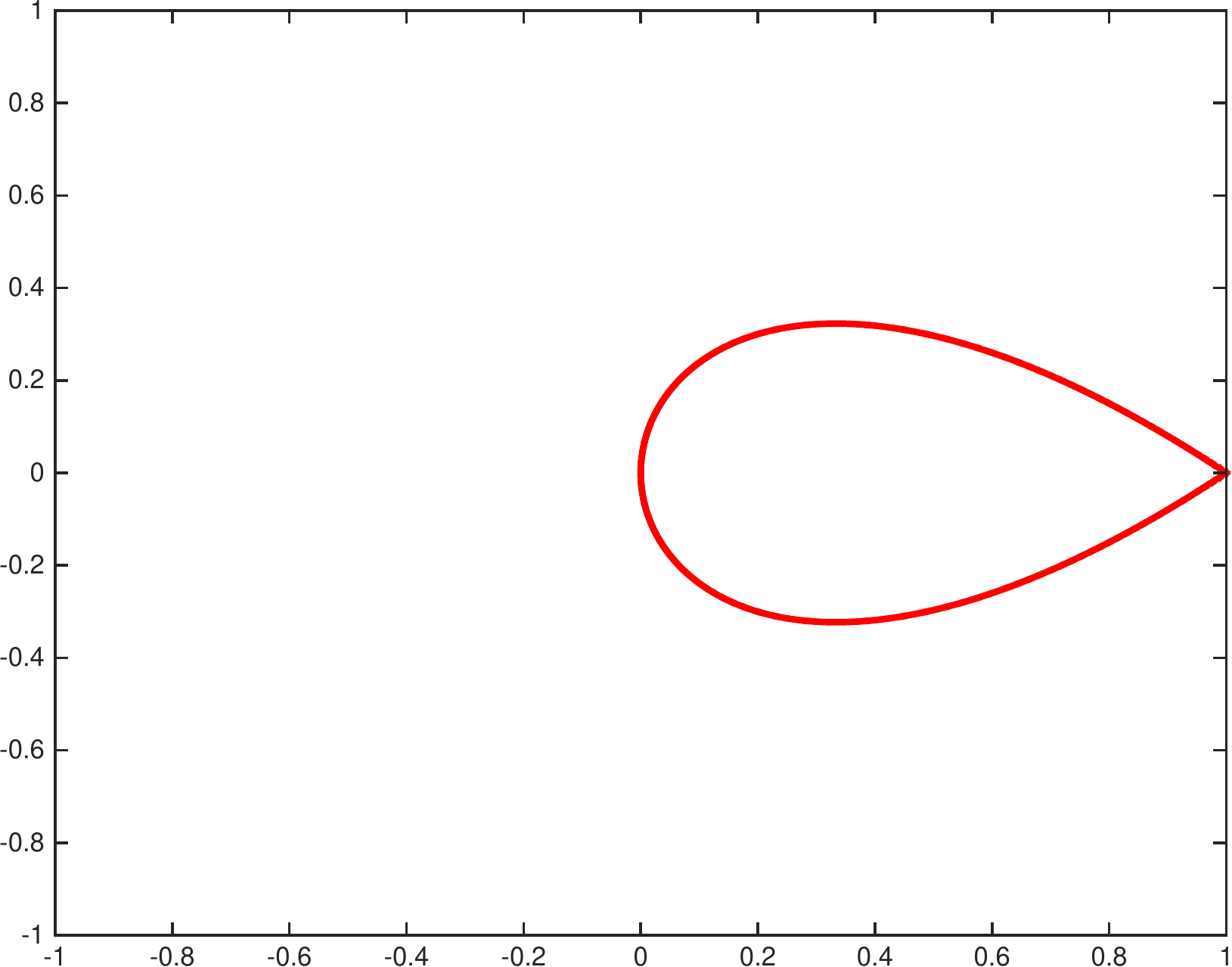}
  \label{fig:sharpint}}
\quad
\subfigure[]{%
     \includegraphics[width=0.47\textwidth]{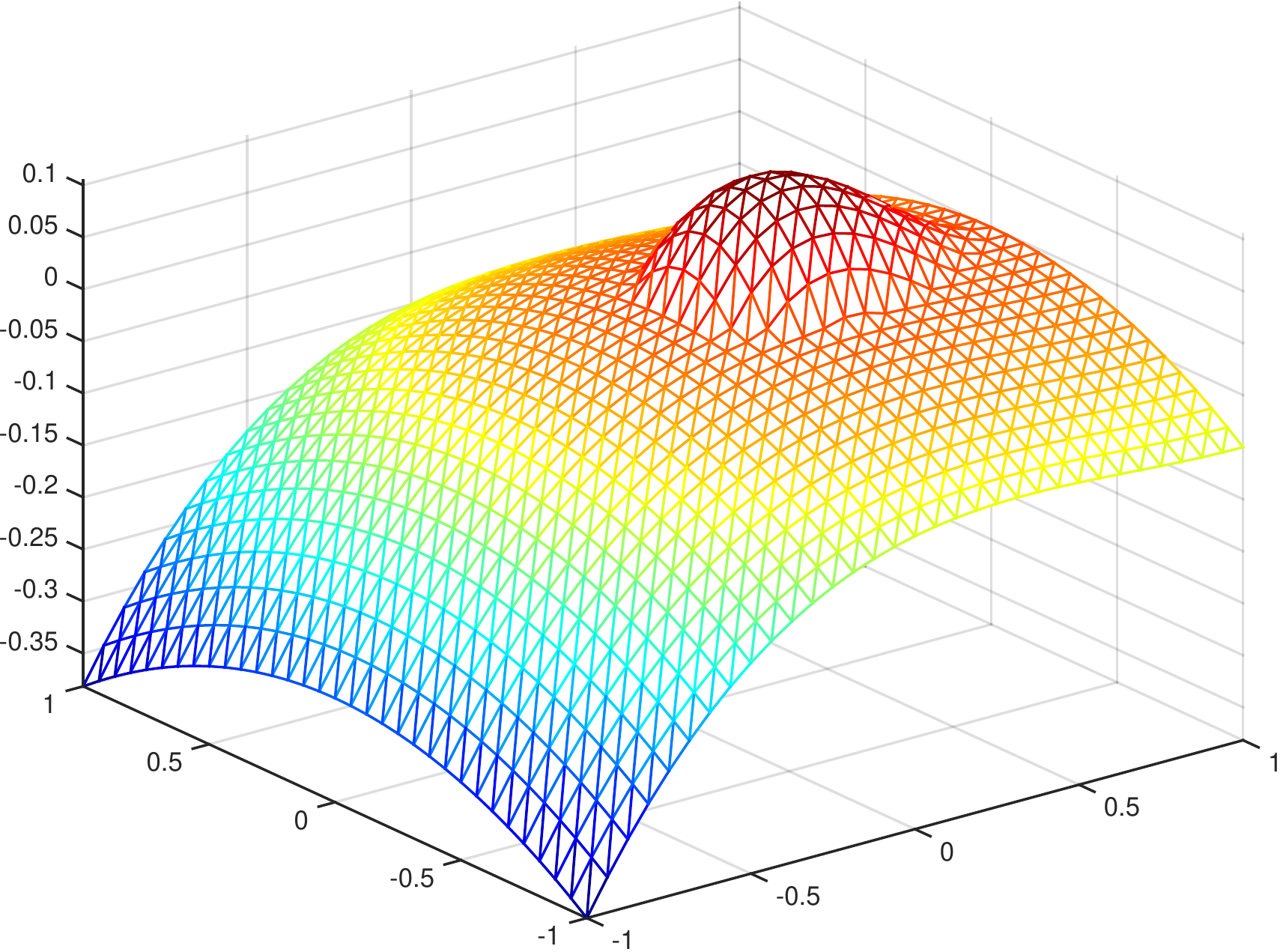}
  \label{fig:sharpsol}}
\caption{Example 2 with $\beta^+=10, \beta^-=1$:  (a)  Shape of interface; (b) Numerical solution of PGIFEM on the coarsest mesh used in Table~\ref{tab:ex2pg}.}
\label{fig:sharp}
\end{figure}

\begin{figure}[ht]
\centering
\subfigure[]{%
     \includegraphics[width=0.45\textwidth]{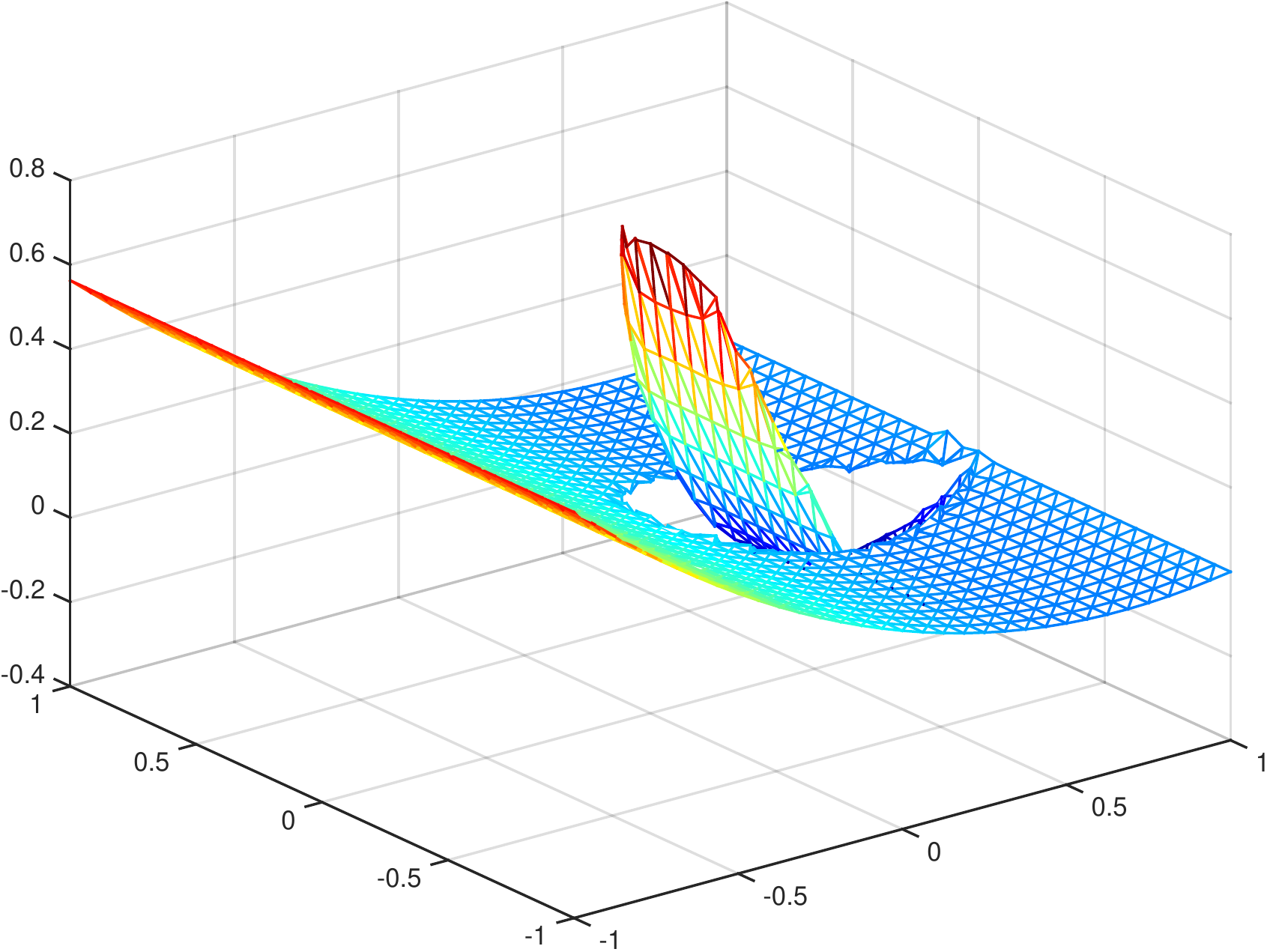}
  \label{fig:sharprx}}
\quad
\subfigure[]{%
     \includegraphics[width=0.45\textwidth]{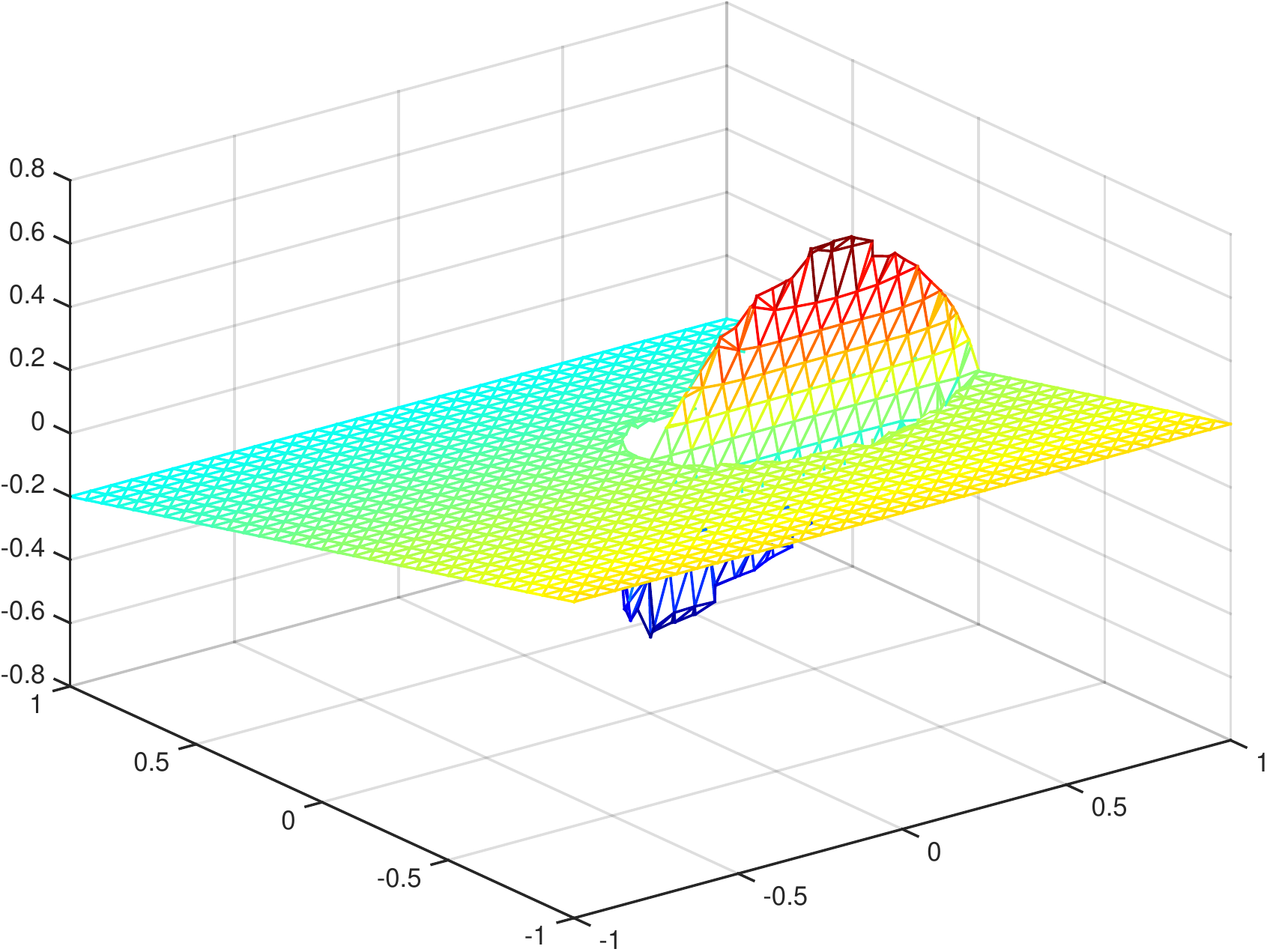}
  \label{fig:sharpry}}
\caption{Plots of  recovered gradient based on PGIFEM for Example 4.2 with $\beta^+=10, \beta^-=1$: (a) $x$-component; (b) $y$-component.}
\label{fig:sharpgrad}
\end{figure}

\begin{table}[htb!]
\centering
\caption{Numerical results of SCIFEM for Example 4.2.}\label{tab:ex2sc}
\begin{tabular}{|c|c|c|c|c|c|c|c|}
\hline
 $N$ & $De$ & order& $D^{i}e$ & order& $D^{r}_re$ & order\\ \hline\hline
 32 &3.04e-02&--&7.01e-03&--&1.12e-02&--\\ \hline
 64 &1.54e-02&0.98&4.63e-03&0.60&4.16e-03&1.42\\ \hline
 128 &7.44e-03&1.05&8.71e-04&2.41&1.02e-03&2.03\\ \hline
 256 &3.71e-03&1.01&3.55e-04&1.30&4.28e-04&1.25\\ \hline
 512 &1.85e-03&1.00&1.27e-04&1.49&1.52e-04&1.49\\ \hline
 1024 &9.24e-04&1.00&4.31e-05&1.56&5.50e-05&1.47\\ \hline
 2048 &4.62e-04&1.00&1.55e-05&1.48&1.99e-05&1.47\\ \hline
\end{tabular}
\end{table}

\begin{table}[htb!]
\centering
\caption{Numerical results of PGIFEM for Example 4.2.}\label{tab:ex2pg}
\begin{tabular}{|c|c|c|c|c|c|c|c|}
\hline
 $N$ & $De$ & order& $D^{i}e$ & order& $D^{r}_re$ & order\\ \hline\hline
 32 &3.57e-02&--&1.98e-02&--&2.09e-02&--\\ \hline
 64 &1.70e-02&1.08&8.37e-03&1.24&8.09e-03&1.37\\ \hline
 128 &7.96e-03&1.09&2.95e-03&1.50&2.73e-03&1.57\\ \hline
 256 &3.82e-03&1.06&9.97e-04&1.57&9.41e-04&1.54\\ \hline
 512 &1.88e-03&1.02&3.72e-04&1.42&3.54e-04&1.41\\ \hline
 1024 &9.32e-04&1.02&1.29e-04&1.53&1.24e-04&1.51\\ \hline
 2048 &4.64e-04&1.01&4.57e-05&1.50&4.31e-05&1.52\\ \hline
\end{tabular}
\end{table}
{\bf Example 4.2.} In this example, we consider the elliptic interface problem \eqref{eq:model} with shape edge as in \cite{JiChenLi2014,KwakWeeChang2010}.
The level set function of the interface is $\phi = -y^2+((x-1)\tan(\theta))^2x$ with $\theta$ being a parameter.  The interface is displayed in
Figure \ref{fig:sharpint}.  The right hand function $f$
is chosen to fit the exact solution $u(x,y) = \phi(x,y)/\beta$.

Numerically we test the case $\beta^- = 1$ and $\beta^+=1000$ when $\theta =40$.  The corresponding numerical results
are shown in Tables \ref{tab:ex2sc} and \ref{tab:ex2pg}, from which one can see that $De$ decays at a optimal rate of $\mathcal{O}(h)$, while $D^ie$ and $D^re$ tend to zero at a superconvergent rate of $\mathcal{O}(h^{1.5})$.     Figure \ref{fig:sharpsol} plots the numerical solution of PGIFEM on the coarsest mesh and
Figure \ref{fig:sharpgrad} shows  the recovered gradient.

\begin{table}[htb!]
\centering
\caption{Numerical results of SCIFEM for Example 4.3. }\label{tab:ex3sc}
\begin{tabular}{|c|c|c|c|c|c|c|c|}
\hline
 $N$ & $De$ & order& $D^{i}e$ & order& $D^{r}_re$ & order\\ \hline\hline
 32 &1.19e+00&--&1.61e-01&--&1.93e-01&--\\ \hline
 64 &5.93e-01&1.00&5.98e-02&1.43&7.24e-02&1.42\\ \hline
 128 &2.96e-01&1.00&2.14e-02&1.48&2.69e-02&1.43\\ \hline
 256 &1.48e-01&1.00&7.80e-03&1.46&9.66e-03&1.48\\ \hline
 512 &7.41e-02&1.00&2.75e-03&1.50&3.49e-03&1.47\\ \hline
 1024 &3.70e-02&1.00&9.84e-04&1.48&1.23e-03&1.51\\ \hline
 2048 &1.85e-02&1.00&3.49e-04&1.50&4.37e-04&1.49\\ \hline
\end{tabular}
\end{table}

\begin{table}[htb!]
\centering
\caption{Numerical results of PGIFEM for Example 4.3.}\label{tab:ex3pg}
\begin{tabular}{|c|c|c|c|c|c|c|c|}
\hline
 $N$ & $De$ & order& $D^{i}e$ & order& $D^{r}_re$ & order\\ \hline\hline
 32 &1.19e+00&--&1.55e-01&--&1.89e-01&--\\ \hline
 64 &5.93e-01&1.00&5.81e-02&1.42&7.19e-02&1.39\\ \hline
 128 &2.96e-01&1.00&2.09e-02&1.48&2.66e-02&1.43\\ \hline
 256 &1.48e-01&1.00&7.61e-03&1.45&9.56e-03&1.48\\ \hline
 512 &7.41e-02&1.00&2.68e-03&1.50&3.45e-03&1.47\\ \hline
 1024 &3.70e-02&1.00&9.61e-04&1.48&1.22e-03&1.51\\ \hline
 2048 &1.85e-02&1.00&3.41e-04&1.50&4.33e-04&1.49\\ \hline
\end{tabular}
\end{table}

{\bf Example 4.3.}  In the example, we consider the elliptic interface problem  \eqref{eq:model} with ellipse interface given by the zero level set of
the function $\phi(x, y) = \frac{x^2}{0.5^2}+\frac{y^2}{0.25^2}-1$ as studied in \cite{JiChenLi2014,KwakWeeChang2010}.
Here, we choose the case of variable coefficient  $\beta(x,y)$ as
\begin{equation*}
\beta(x,y) =
\left\{
\begin{array}{lcc}
    1+0.5(x^2-xy+y^2) &  \text{if } (x,y)\in \Omega^-, \\
   1 &   \text{if } (x,y)\in \Omega^+.
\end{array}
\right.
\end{equation*}
The right hand side function $f$ and boundary condition are given by the exact solution $u(x,y) = \phi(x,y)/\beta(x,y)$.

Tables \ref{tab:ex3sc} and \ref{tab:ex3pg} list the numerical errors, which provide a verification of the $\mathcal{O}(h)$ convergence for semi-$H^1$ error,  and $\mathcal{O}(h^{1.5})$ supercloseness and superconvergence.

\begin{figure}[ht]
\centering
\subfigure[]{%
     \includegraphics[width=0.43\textwidth]{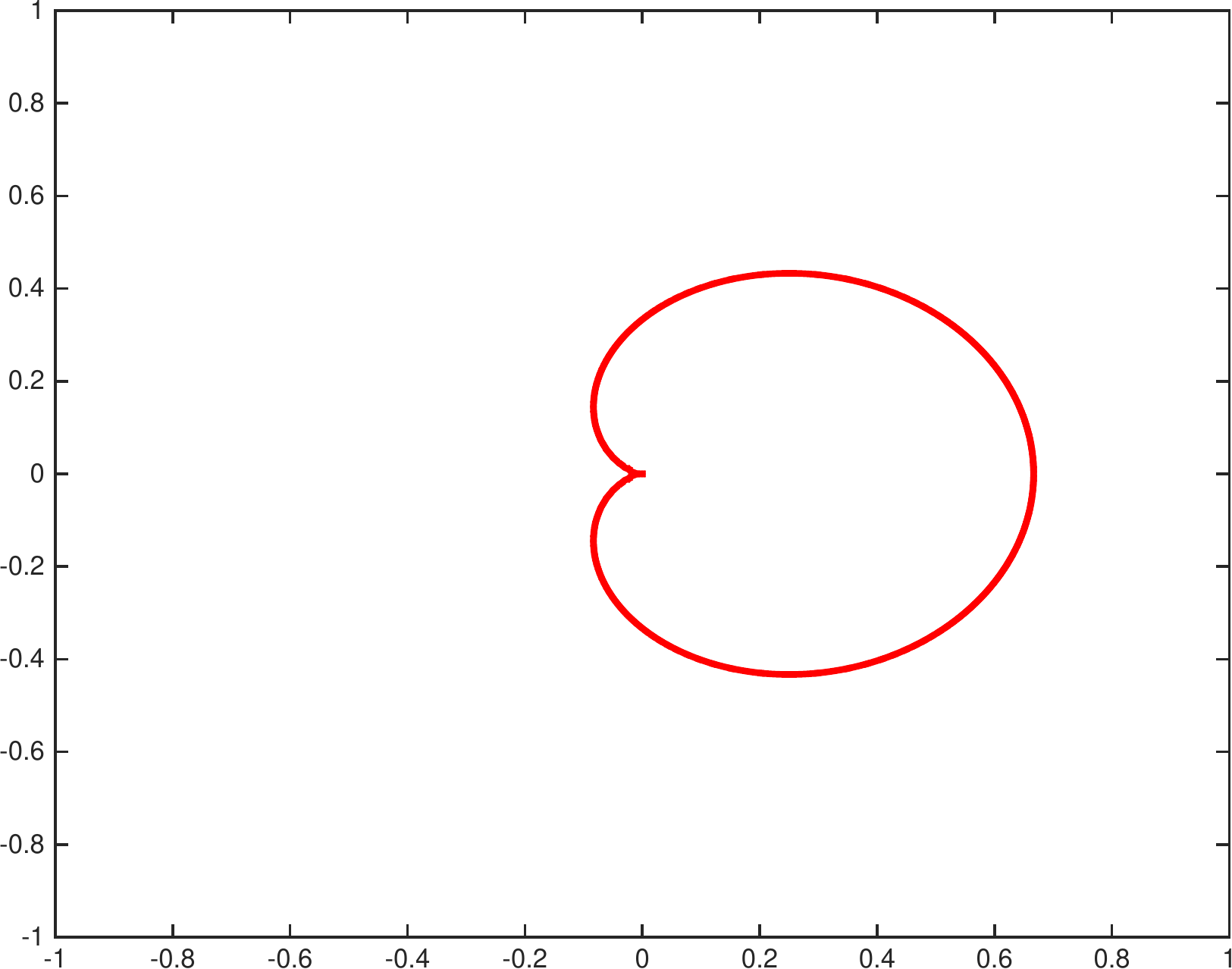}
  \label{fig:cardioidint}}
\quad
\subfigure[]{%
     \includegraphics[width=0.47\textwidth]{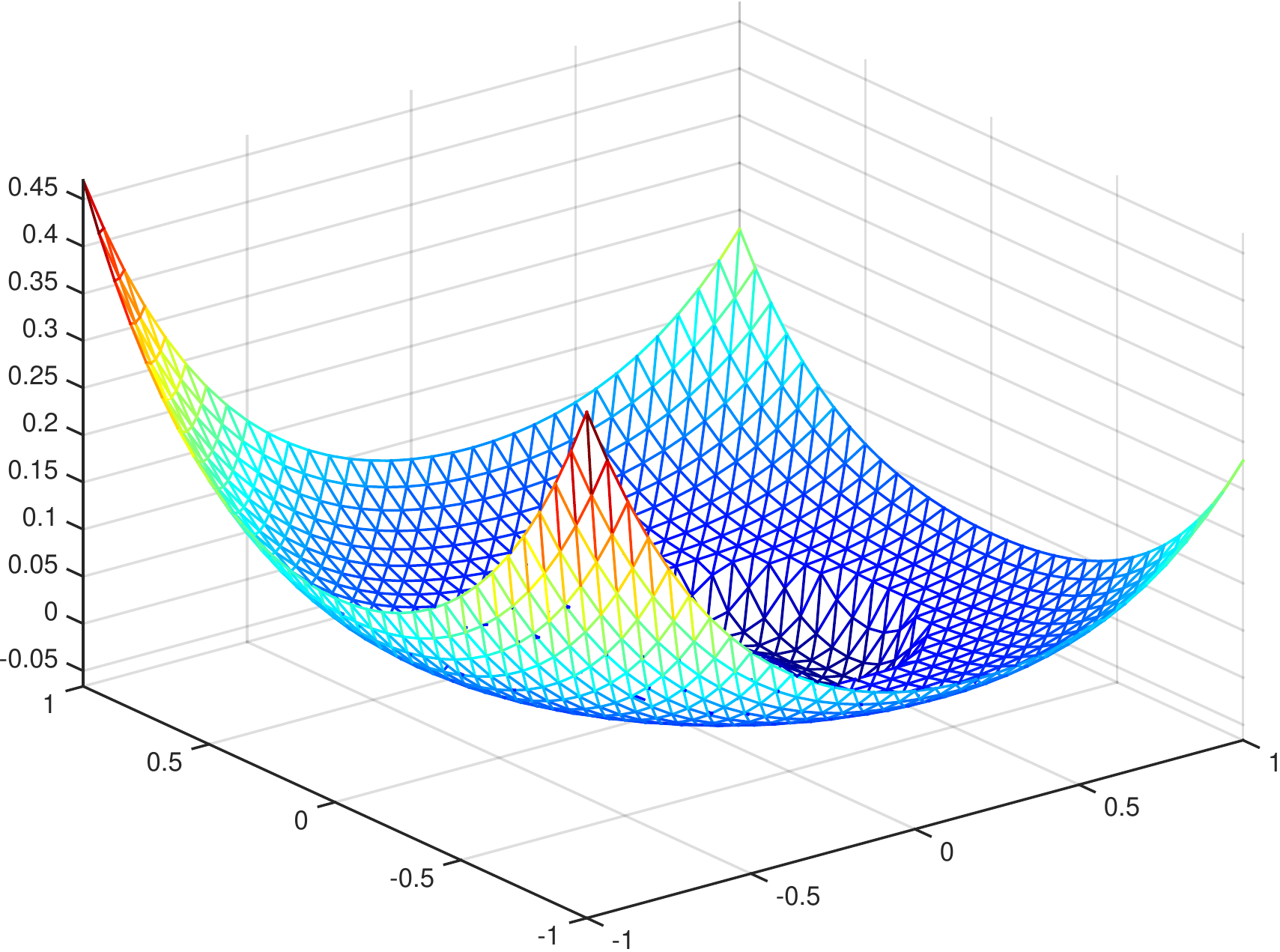}
  \label{fig:cardioidsol}}
\caption{Example 4.4 with $\beta^+=10, \beta^-=1$:  (a)  Shape of interface; (b) Numerical solution of PGIFEM on the coarsest mesh used in Table~\ref{tab:ex4pg}.}
\label{fig:cardioit}
\end{figure}

\begin{figure}[ht]
\centering
\subfigure[]{%
     \includegraphics[width=0.45\textwidth]{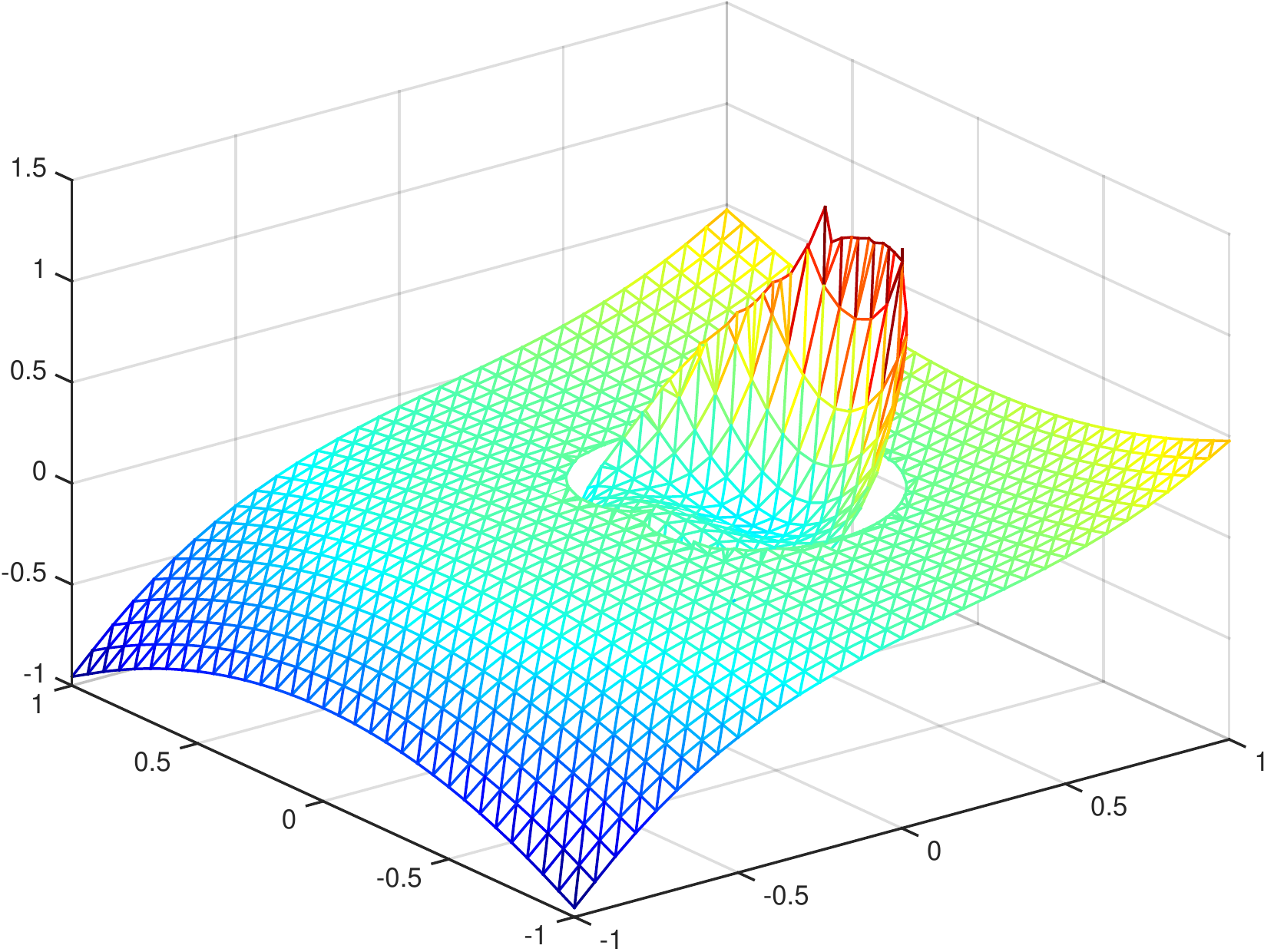}
  \label{fig:circlecx}}
\quad
\subfigure[]{%
     \includegraphics[width=0.45\textwidth]{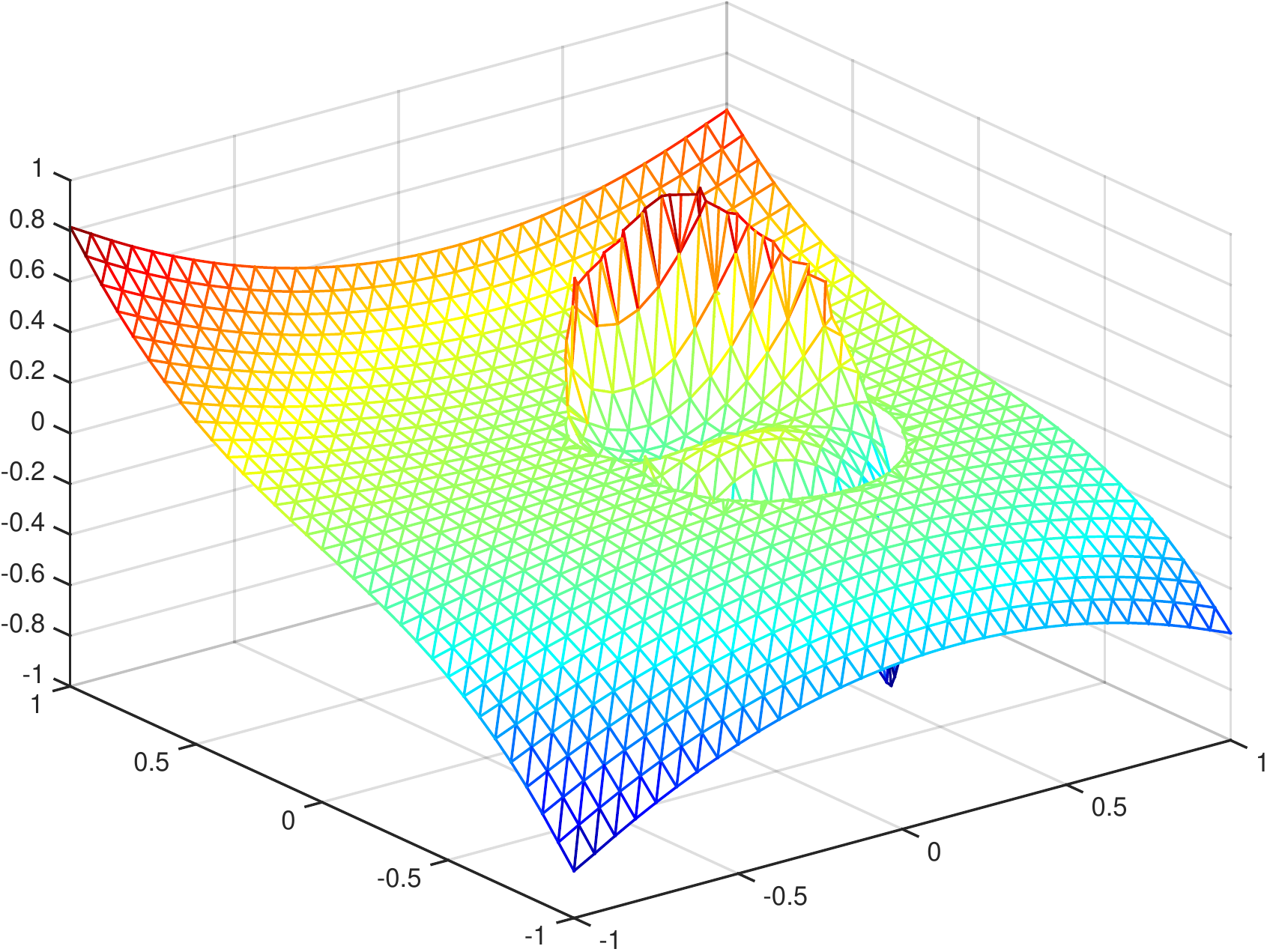}
  \label{fig:circlery}}
\caption{Plots of  recovered gradient based on PGIFEM for Example 4.4 with $\beta^+=10, \beta^-=1$: (a) $x$-component; (b) $y$-component.}
\label{fig:cardioidgrad}
\end{figure}

\begin{table}[htb!]
\centering
\caption{Numerical results of SCIFEM for Example 4.4. }\label{tab:ex4sc}
\begin{tabular}{|c|c|c|c|c|c|c|c|}
\hline
 $N$ & $De$ & order& $D^{i}e$ & order& $D^{r}_re$ & order\\ \hline\hline
 32 &5.59e-02&--&9.84e-03&--&2.51e-02&--\\ \hline
 64 &2.88e-02&0.96&4.09e-03&1.27&7.98e-03&1.65\\ \hline
 128 &1.47e-02&0.98&1.57e-03&1.38&2.30e-03&1.80\\ \hline
 256 &7.39e-03&0.99&5.72e-04&1.46&7.21e-04&1.67\\ \hline
 512 &3.71e-03&0.99&2.05e-04&1.48&2.41e-04&1.58\\ \hline
 1024 &1.86e-03&1.00&7.25e-05&1.50&8.99e-05&1.42\\ \hline
 2048 &9.31e-04&1.00&2.54e-05&1.51&3.15e-05&1.51\\ \hline
\end{tabular}
\end{table}

\begin{table}[htb!]
\centering
\caption{Numerical results of PGIFEM for Example 4.4.  }\label{tab:ex4pg}
\begin{tabular}{|c|c|c|c|c|c|c|c|}
\hline
 $N$ & $De$ & order& $D^{i}e$ & order& $D^{r}_re$ & order\\ \hline\hline
 32 &6.09e-02&--&2.48e-02&--&4.06e-02&--\\ \hline
 64 &3.01e-02&1.01&9.06e-03&1.45&1.36e-02&1.58\\ \hline
 128 &1.50e-02&1.01&3.32e-03&1.45&4.25e-03&1.68\\ \hline
 256 &7.48e-03&1.00&1.16e-03&1.51&1.45e-03&1.55\\ \hline
 512 &3.73e-03&1.00&4.13e-04&1.49&4.87e-04&1.57\\ \hline
 1024 &1.87e-03&1.00&1.44e-04&1.52&1.74e-04&1.49\\ \hline
 2048 &9.32e-04&1.00&5.11e-05&1.49&6.11e-05&1.51\\ \hline
\end{tabular}
\end{table}

{\bf Example 4.4.}  In this example, we consider the interface problem  \eqref{eq:model} with a cardioid interface as in \cite{HouLiu2005}.
The interface curve $\Gamma$ is  the zero level of the function
\begin{equation*}
\phi(x,y) = (3(x^2+y^2)-x)^2-x^2-y^2,
\end{equation*}
as shown Figure \ref{fig:cardioidint}.    We choose  the exact solution $u(x,y) = \phi(x,y)/\beta(x,y)$, where
\begin{equation*}
\beta(x,y) =
\left\{
\begin{array}{lcc}
    xy+3 &  \text{if } (x,y)\in \Omega^-, \\
   100 &   \text{if } (x,y)\in \Omega^+.
\end{array}
\right.
\end{equation*}

As pointed in \cite{HouLiu2005}, the  difficulty of the problem is that the interface  is not even Lipschitz-continuous and has a singular point at
the origin. Figure \ref{fig:cardioidsol} plots the numerical solution of PGIFEM and Figure \ref{fig:cardioidgrad} shows the recovered gradient.   The
numerical errors are given in Tables \ref{tab:ex4sc} and \ref{tab:ex4pg}, from which, one can also observe the optimal convergence
and superconvergence for both SCIFEM and PGIFEM even though the interface is not Lipschitz-continuous.

\section{Conclusion}

In this paper, we develop gradient recovery methods for both symmetric consistent immersed finite method and Petrov-Galerkin immersed finite element method. Theoretically, we prove that the proposed gradient recovery operator has consistency, localization, and boundedness properties.
The superconvergence of recovered gradient is confirmed by four numerical examples using both piecewise
constant and piecewise variable diffusion coefficients. Moreover, we numerically observe the supercloseness between
immersed finite element solution and the linear interpolation of exact solution. Compared to body-fitted mesh-based gradient recover methods, the proposed gradient recovery methods provide a uniform way of recovering gradient on regular meshes.

\section*{Acknowledgement}
This work was partially supported by the NSF grant DMS-1418936, KI-Net NSF RNMS
grant 1107291, and Hellman Family Foundation Faculty Fellowship, UC Santa Barbara.
Part of work was done during the visit of both authors to Beijing Computational Science Research Center,
and we really appreciate their hospitality.

\bibliography{mybibfile}

\end{document}